\documentclass[12pt]{amsart}
\usepackage{amsmath,amssymb,amsthm,amscd}
\newtheorem{thm}{Theorem}

\newtheorem{lem}[thm]{Lemma}

\newtheorem*{rem}{Remark}
\newtheorem*{ack}{Acknowledgements}

\newcommand{\SL}{{\rm SL}}
\newcommand{\GL}{{\rm GL}}
\newcommand{\C}{\mathbb{C}}
\newcommand{\Q}{\mathbb{Q}}
\newcommand{\Z}{\mathbb{Z}}
\newcommand{\R}{\mathbb{R}}

\begin{document}

\title{Explicit bounds for sums of squares}
\author{Jeremy Rouse}
\address{Department of Mathematics, Wake Forest University,
  Winston-Salem, NC 27109}
\email{rouseja@wfu.edu}
\thanks{The author was supported by NSF grant DMS-0901090}
\begin{abstract}
For an even integer $k$, let $r_{2k}(n)$ be the number of representations
of $n$ as a sum of $2k$ squares. The quantity $r_{2k}(n)$ is appoximated
by the classical singular series $\rho_{2k}(n) \asymp n^{k-1}$. Deligne's
bound on the Fourier coefficients of Hecke eigenforms gives that
$r_{2k}(n) = \rho_{2k}(n) + O(d(n) n^{\frac{k-1}{2}})$. We determine
the optimal implied constant in this estimate provided that
either $k/2$ or $n$ is odd. The proof requires a delicate positivity
argument involving Petersson inner products.
\end{abstract}

\maketitle

\section{Introduction and Statement of Results}

In Hardy's book on Ramanujan \cite{Hardy}, he states the following
(Chapter 9, pg. 132).
\begin{quote}
The problem of the representations of an integer $n$ as the sum of a
given number $k$ of integral squares is one of the most celebrated in the
theory of numbers. Its history may be traced back to Diophantus, but
begins effectively with Girard's (or Fermat's) theorem that a prime $4m+1$
is the sum of two squares. Almost every arithmetician of note since Fermat
has contributed to the solution of the problem, and it has its puzzles
for us still.
\end{quote}

If $n$ is a non-negative integer, let
\[
  r_{s}(n) = \# \{ (x_{1}, x_{2}, \ldots, x_{s}) \in \Z^{s}
  : x_{1}^{2} + x_{2}^{2} + \cdots + x_{s}^{2} = n \}
\]
be the number of representations of $n$ as a sum of $s$ squares. 

The classical work that Hardy refers to includes the formulas of Jacobi
giving the following exact formulas. Let $n$ be a positive integer
and write $n = 2^{\alpha} m$, where $m$ is odd. Then
\[
  r_{4}(n) = \begin{cases}
    8 \sigma_{1}(m) & \text{ if } \alpha = 0\\
    24 \sigma_{1}(m) & \text{ if } \alpha \geq 1,
\end{cases}
\quad
  r_{8}(n) = \begin{cases}
   16 \sigma_{3}(m) & \text{ if } \alpha = 0\\
   16 \cdot \frac{2^{3 \alpha + 3} - 15}{7} \sigma_{3}(m) & \text{ if } \alpha
  \geq 1.
\end{cases}
\]
The search for higher exact formulas (each
involving more complicated arithmetic functions) for was carried out
by many mathematicians. Glaisher \cite{Glaisher} and Rankin \cite{Rankin} were 
interested in these formulas where the arithmetic functions involved were 
multiplicative.

In a different direction, Hardy \cite{Hardy2} and Mordell \cite{Mordell} 
applied the circle method to give an approximation
\[
  r_{s}(n) = \rho_{s}(n) + R_{s}(n)
\]
where $\rho_{s}(n)$ is the ``singular series'' and $R_{s}(n)$ is an
error term. Here $\rho_{s}(n)$ can be expressed as a divisor
sum if $s$ is even, and $\rho_{s}(n) \asymp n^{\frac{s}{2} - 1}$
provided $s > 4$. The contribution $R_{s}(n)$ is more mysterious,
and Deligne's proof of the Weil conjectures (see \cite{Del}) implies an 
estimate of the form
\begin{equation}
\label{errorbound}
  R_{s}(n) = O(d(n) n^{\frac{s}{4} - \frac{1}{2}})
\end{equation}
provided $s$ is even. The phenomena of exact formulas for $r_{s}(n)$
of the form $r_{s}(n) = \rho_{s}(n)$ only occurs for small $s$. In 
\cite{Rankin2}, Rankin shows that $R_{s}(n)$ is identically zero if and only 
if $s \leq 8$. Exact formulas of a different nature were given by
Milne in \cite{Milne} when $s = 4n^{2}$ and $s = 4n(n+1)$. 

The problem we study is the implied constant in equation
\eqref{errorbound} above. This is a natural question, and in
\cite{JR1,JR2,JR3}, the author has studied the corresponding problem for 
powers of the $\Delta$ function, $p$-core partitions (joint work with Byungchan
Kim), and arbitrary level 1 cusp forms (joint work with Paul Jenkins), 
respectively. To prove their now famous ``290-theorem'' Bhargava and Hanke 
\cite{BH} compute this implied constant for about 6000 quadratic forms in four 
variables and use this to determine precisely which integers these forms
represent.

Returning to our problem, if $s = 2k$ and $k$ is even, we have that
\[
  \rho_{2k}(n) = \frac{2k (-1)^{k/2+1}}{(2^{k} - 1) B_{k}}
  \left(\sigma_{k-1}(n) + (-1 + (-1)^{k/2+1}) \sigma_{k-1}(n/2)
  + (-1)^{k/2} 2^{k} \sigma_{k-1}(n/4)\right),
\]
where $B_{k}$ is the $k$th Bernoulli number and $\sigma_{k-1}(n)$
is the sum of the $k-1$st powers of the positive integer divisors of $n$ (and 
is hence zero if $n$ is not an integer). Our main result is the following.

\begin{thm}
\label{main}
Suppose that $k$ is even. If either $k/2$ is odd or $n$ is odd, then we have
\[
  \left|r_{2k}(n) - \rho_{2k}(n)\right| \leq 
\left(4k + \frac{2k (-1)^{k/2}}{(2^{k} - 1) B_{k}}\right) 
d(n) n^{\frac{k-1}{2}}.
\]
\end{thm}

\begin{rem}
If $2k = 4$ or $2k = 8$, the right hand side is zero, and we recover
the exact formulas of Jacobi. For arbitrary even $k$, we have
$r_{2k}(1) = 4k$ and $\rho_{2k}(1) = \frac{2k (-1)^{k/2+1}}{(2^{k} - 1) B_{k}}$.
Thus, the inequality above becomes an equality when $n = 1$. This
shows that the implied constant
\[
  4k + \frac{2k (-1)^{k/2}}{(2^{k} - 1) B_{k}}
\]
in \eqref{errorbound} is best possible. The error term is smaller than the
main term provided $n \gg k^{2}$.
\end{rem}

Our approach to proving Theorem~\ref{main} is as follows. If
\[
  \theta(z) = 1 + 2 \sum_{n=1}^{\infty} q^{n^{2}}, \quad q = e^{2 \pi i z}
\]
is the classical Jacobi theta function, then 
\[
  \theta^{2k}(z) = \sum_{n=0}^{\infty} r_{2k}(n) q^{n}
\]
is a modular form of weight $k$ on $\Gamma_{0}(4)$. If $k$ is even, we can 
decompose
\begin{equation}
\label{decomp}
  \theta^{2k}(z) = a_{1} E_{k}(z) + a_{2} E_{k}(2z) + a_{3} E_{k}(4z)
  + \sum_{i} c_{i} g_{i}(z) + \sum_{i} d_{i} g_{i}(2z) + \sum_{i} e_{i}
  g_{i}(4z)
\end{equation}
where
\[
  E_{k}(z) = 1 - \frac{2k}{B_{k}} \sum_{n=1}^{\infty} \sigma_{k-1}(n) q^{n}
\]
is the classical level 1 Eisenstein series, and the $g_{i}(z)$
are normalized newforms of level $1$, $2$, or $4$. We prove the following.

\begin{thm}
\label{positivity}
Assume the notation above. Then for all $i$, $c_{i} \geq 0$.
\end{thm}

Theorem~\ref{positivity} allows us to read off
\[
  \sum_{i} |c_{i}| = \sum_{i} c_{i} = 4k
  + \frac{2k (-1)^{k/2}}{(2^{k} - 1) B_{k}}
\]
from the coefficient of $q$ on both sides of \eqref{decomp}, using
that $a_{1} = \frac{(-1)^{k/2}}{(2^{k} - 1) B_{k}}$.

To prove Theorem~\ref{positivity} we use properties of the Petersson
inner product on $M_{k}(\Gamma_{0}(4))$ (see Section~\ref{prelim} for
precise definitions). If $g_{i}(z)$ is a newform of level $4$, then
$g_{i}(z)$ is orthogonal to every other term in the expansion \eqref{decomp}. 
It follows that
\[
  \langle \theta^{2k}, g_{i} \rangle = c_{i} \langle g_{i}, g_{i} \rangle.
\]
It suffices to prove that $\langle \theta^{2k}, g_{i} \rangle \geq 0$.
This Petersson inner product consists of a contribution from each of the
three cusps of $\Gamma_{0}(4)$: $\infty$, $0$, and $1/2$. The contribution
from $\infty$ is
\[
  \frac{2}{(4 \pi)^{k}} \sum_{n=1}^{\infty} \frac{r_{2k}(n) a(n)}{n^{k-1}}
  \int_{4 \pi n}^{\infty} u^{k-2} e^{-u} \, du.
\]
Here $g_{i}(z) = \sum_{n=1}^{\infty} a(n) q^{n}$. Our approach is to show that
the main term in the above sum comes from $n = 1$. If $n$ is fixed,
$r_{2k}(n)$ is a polynomial of degree $2k$ in $n$. We compute these 
polynomials explicitly, and use this to the bound the terms
when $2 \leq n \leq 2500$. Next, we use a simple induction bound on $r_{2k}(n)$
to show that the terms with $2500 \leq n \leq \frac{k}{2 \pi} \log(k)$
are small enough. Finally, we use the exponential decay of 
$\int_{4 \pi n}^{\infty} u^{k-2} e^{-u} \, du$ when 
$n \geq \frac{k}{2 \pi} \log(k)$.

The cusp at zero behaves in an essentially identical way to the cusp at 
$\infty$, and the contribution from the cusp at $1/2$ is very small,
since $\theta(z)$ vanishes there.

\begin{rem}
  This result can be thought of as a refined form of the circle
  method.  The Eisenstein series is the contribution of the major
  arcs, while the Deligne's result, and the bounds we give on the
  constants $c_{i}$ can be thought of as explicit, uniform minor arc
  estimates. Further, it is plausible that the Fourier coefficients of
  distinct newforms are independent (an assertion that could be
  justified under the assumption of the holomorphy of certain
  Rankin-Selberg convolutions). Combining this with the recent proof
  of the Sato-Tate conjecture (see \cite{BGHT}) suggests that for any
  $\epsilon > 0$, there are infinitely many primes $p$ so that
\[
  |r_{2k}(p) - \rho_{2k}(p)| > \left(4k + \frac{2k (-1)^{k/2}}{(2^{k} - 1) B_{k}} - \epsilon\right) d(p) p^{\frac{k-1}{2}}.
\]
\end{rem}

\begin{rem}
The proof gives more detailed information about the constants $c_{i}$
in \eqref{decomp}. In particular, if $g_{i}(z)$ is a newform of level $4$ and 
$k \equiv 2 \pmod{4}$, then
\[
  c_{i} = 16k \cdot \frac{(k-2)!}{(4 \pi)^{k} \langle g_{i}, g_{i} \rangle}
  (1 + O(\alpha^{k}))
\]
where $\alpha \approx 0.918$. If $k \equiv 0 \pmod{4}$, then $c_{i} = 0$.
Similar, but more complicated results are true for the constants $c_{i}$
associated with level 1 and level 2 newforms.
\end{rem}

An outline of the paper is as follows. In Section~\ref{prelim} we
give precise definitions and review necessary background information. In
Section~\ref{lemmas} we prove a number of auxiliary results that will
be used in the proof of Theorem~\ref{positivity}. In Section~\ref{mainsec},
we prove Theorem~\ref{positivity} and use this to deduce
Theorem~\ref{main}. Finally, in Section~\ref{finalsec}, we
address other values of $k$ and $n$.

\begin{ack}
The author used Magma \cite{Magma} version 2.17 for computations. 
\end{ack}

\section{Background}
\label{prelim}

In this section we give definitions and review necessary background.
For $N \geq 1$, let $M_{k}(\Gamma_{0}(N))$ denote the $\C$-vector
space of modular forms of weight $k$ on $\Gamma_{0}(N) :=
\left\{ \left( \begin{matrix} a & b \\ c & d \end{matrix} \right)
\in \SL_{2}(\Z) : N | c \right\}$. Let $S_{k}(\Gamma_{0}(N))$ denote
the subspace of cusp forms. 

If $f$ is a modular form of weight $k$, and $\alpha = \left[ \begin{matrix}
a & b \\ c & d \end{matrix} \right] \in \GL_{2}(\Q)$ and has positive
determinant, define the usual slash operator by
\[
  f | \alpha = (ad-bc)^{k/2} (cz+d)^{-k} f\left(\frac{az+b}{cz+d}\right).
\]
For a positive integer $d$, define the operator $V(d)$ by
$f(z) | V(d) = f(dz)$. It is well-known (see \cite{Iwa}, pg. 107 for a proof)
that $V(d)$ maps $M_{k}(\Gamma_{0}(N))$ to $M_{k}(\Gamma_{0}(dN))$
and $S_{k}(\Gamma_{0}(N))$ to $S_{k}(\Gamma_{0}(dN))$. For a positive integer 
$d$, define the operator $U(d)$ by
\[
  \sum_{n=0}^{\infty} a(n) q^{n} | U(d)
  = \sum_{n=0}^{\infty} a(dn) q^{n}.
\]
If $d | N$, then $U(d)$ maps $M_{k}(\Gamma_{0}(N))$ to itself and
$S_{k}(\Gamma_{0}(N))$ to itself. If $p$ is a prime with $p \nmid N$,
define the usual Hecke opeator $T(p)$ by $T(p) = U(p) + p^{k-1} V(p)$.

If $f, g \in M_{k}(\Gamma_{0}(N))$ and at least one of $f$ or $g$ is a cusp 
form, let
\[
  \langle f, g \rangle = \frac{3}{\pi [\SL_{2}(\Z) : \Gamma_{0}(N)]}
  \iint_{\mathbb{H}/\Gamma_{0}(N)} f(x+iy) \overline{g(x+iy)} y^{k} \, \frac{dx \, dy}{y^{2}}
\]
denote the usual Petersson inner product. If $p \nmid N$, then the
Hecke operators $T(p)$, acting on $S_{k}(\Gamma_{0}(N))$, are
self-adjoint with respect to the Petersson inner product. Moreover, if
$\alpha \in \GL_{2}(\Q)$ and has positive determinant, then $\langle f
| \alpha, g | \alpha \rangle = \langle f, g \rangle$.

Let $S_{k}^{{\rm new}}(\Gamma_{0}(N))$ denote the orthogonal
complement under this inner product of the space spanned by all forms
\[
  f(z) | V(d), \text{ where } f(z) \in S_{k}(\Gamma_{0}(M)),
\]
and we have $M | N$, $M < N$, and $d$ is a divisor of $N/M$. A \emph{newform}
of level $N$ is a form 
\[
  f(z) = \sum_{n=1}^{\infty} a(n) q^{n} \in S_{k}^{{\rm new}}(\Gamma_{0}(N))
\] 
that is a simultaneous eigenform of the Hecke operators $T(p)$, normalized so
that $a(1) = 1$. We have the Deligne bound
\[
  |a(n)| \leq d(n) n^{\frac{k-1}{2}}
\]
where $d(n)$ is the number of divisors of $n$ (for a detailed proof of this
inequality, see the new book by Brian Conrad \cite{Conrad}). 
A newform $f(z)$ of level $N$ is also an eigenform of the Atkin-Lehner
operator $W_{N} = \left[ \begin{matrix} 0 & -1 \\ N & 0 \end{matrix} \right]$.
This operator commutes with the Hecke operators $T(p)$ for primes $p \nmid N$.
One has more information about the coefficient $a(p)$ if $p | N$. If $N = p$, 
then $a(p) = -\lambda p^{\frac{k}{2} - 1}$, where $\lambda$ is the eigenvalue of 
$f$ under $W_{N}$. If $p^{2} | N$, then $a(p) = 0$ (see \cite{AL}, Theorem 3).

The multiplicity-one theorem states that the joint eigenspaces of all
$T(p)$ (with $p \nmid N$) in $S_{k}^{{\rm new}}(\Gamma_{0}(N))$ are 
one-dimensional. It follows from this, and the 
self-adjointness of the Hecke opeators, that if $f_{1}$ and $f_{2}$ are
two distinct newforms, then $\langle f_{1}, f_{2} \rangle = 0$. It is known
(see Section 5.11 of \cite{DiaShu}) that the Eisenstein series
$E_{k}(z)$ (and $E_{k}(z) | V(d)$) are orthogonal to cusp forms under the
Petersson inner product.

Finally, let $\eta(z)$ denote as usual the Dedekind eta function
\[
  \eta(z) = q^{1/24} \prod_{n=1}^{\infty} (1-q^{n}), \quad q = e^{2 \pi i z}.
\]
We have the following well-known identities:
\begin{align*}
  \theta(z) &= \frac{\eta^{5}(2z)}{\eta^{2}(z) \eta^{2}(4z)}\\
  \frac{\eta^{8}(4z)}{\eta^{4}(2z)} &= \sum_{n=0}^{\infty}
  \sigma(2n+1) q^{2n+1}\\
  (2z+1)^{-2} \theta^{4}\left(\frac{z}{2z+1}\right)
  &= 16 \frac{\eta^{8}(4z)}{\eta^{4}(2z)}
\end{align*}
(see the exercises on page 145 of \cite{Koblitz}, solutions are
on page 234).  

\section{Preliminary results}
\label{lemmas}

In this section we prove three lemmas that will be used in the proof
of the main results. Our first lemma proves some simple bounds on $r_{s}(n)$.

\begin{lem}
\label{simplebound}
\begin{enumerate}
\item Suppose that $n$ is a non-negative integer. There are non-negative
constants $c_{i,n}$ ($0 \leq i \leq n$) so that
\[
  r_{s}(n) = \sum_{i=0}^{n} c_{i,n} \binom{s}{i} \text{ for all } s \geq 0.
\]
\item If $n$ is fixed, $\frac{r_{2s}(n)}{n^{\frac{s-1}{2}}}$ is a decreasing
function of $s$, provided $2s \geq n + \frac{n}{\sqrt[4]{n} - 1}$.
\item If $n$ is a positive integer and $s \geq 6$, then
\[
  r_{s}(n) \leq \frac{3 (4.11)^{s}}{25 \sqrt{s!}} (n+s)^{\frac{s}{2} - 1}.
\]
\end{enumerate}
\end{lem}
\begin{proof}
We prove the first statement by strong induction on $n$. For $n = 0$, we have
$r_{s}(0) = 1 = 1 \cdot \binom{s}{0}$. Thus, $c_{0,0} = 1$ and the
result holds.

Assume the result is true for all $m < n$. Let $t$ be a positive integer
with $t \leq s$. Then
\begin{align*}
  r_{t}(n) - r_{t-1}(n) &= 2 \sum_{r=1}^{\lfloor \sqrt{n} \rfloor}
  r_{t-1}(n-r^{2})\\
  &= 2 \sum_{r=1}^{\lfloor \sqrt{n} \rfloor}
  \sum_{i=0}^{n-r^{2}} c_{i,n-r^{2}} \binom{t-1}{i}.
\end{align*}
Summing both sides over all $t$, $1 \leq t \leq s$ and using
that $\sum_{t=1}^{s} \binom{t-1}{i} = \binom{s}{i+1}$ gives
\begin{align*}
  r_{s}(n) &= \sum_{r=1}^{\lfloor \sqrt{n} \rfloor}
  \sum_{i=0}^{n-r^{2}} 2 c_{i,n-r^{2}} \binom{s}{i+1}\\
  &= 2 \sum_{i=1}^{n} \left(\sum_{r=1}^{\lfloor \sqrt{n-i} \rfloor}
  c_{i-1,n-r^{2}}\right) \binom{s}{i}.
\end{align*}
Since the $c_{i-1,n-r^{2}}$ are non-negative, by the induction hypothesis,
it follows that their sum is non-negative, and this proves that the result is 
true for $n$. 

To prove the second statement, it suffices to prove that each term in the
expression
\[
  \frac{r_{2s}(n)}{n^{\frac{s-1}{2}}}
  = \sum_{i=0}^{n} c_{i,n} \frac{\binom{2s}{i}}{n^{\frac{s-1}{2}}}
\]
is a decreasing function of $s$. Let $f(s) = \binom{2s}{i} \cdot n^{(1-s)/2}$.
Then,
\begin{align*}
  \frac{f(s+1)}{f(s)} &= \frac{1}{\sqrt{n}} \cdot \frac{(2s+2)(2s+1)}{(2s+2-i) (2s+1-i)}\\
  &\leq \frac{1}{\sqrt{n}} \frac{(2s+2)(2s+1)}{(2s+2-n)(2s+1-n)}\\
  &< \frac{1}{\sqrt{n}} \left(1 + \frac{n}{2s-n}\right)^{2}.
\end{align*}
This is a decreasing function of $s$, and if we take
$s = n + \frac{n}{\sqrt[4]{n} - 1}$, then $2s - n = \frac{n}{\sqrt[4]{n} - 1}$
and so
\[
  \frac{1}{\sqrt{n}} \left(1 + \frac{n}{2s-n}\right)^{2}
  = \frac{1}{\sqrt{n}} \left(1 + (\sqrt[4]{n} - 1)\right)^{2} = 1.
\]
This proves that $f(s+1) < f(s)$, as desired.

We prove the third statement by induction on $s$. Our base case is $s
= 6$ and in this case, we use the exact formula 
\[
  r_{6}(n) = \sum_{d | n} d^{2} \left(-4 \chi_{-1}(d) + 16 \chi_{-1}(n/d)\right),
\]
where
\[
  \chi_{-1}(n) = \begin{cases}
     1 & \text{ if } n \equiv 1 \pmod{4}\\
    -1 & \text{ if } n \equiv 3 \pmod{4}\\
     0 & \text{ if } n \text{ is even. }
\end{cases}
\]
We rewrite this as
\[
  r_{6}(n) = n^{2} \sum_{d | n} \frac{16 \chi_{-1}(n/d) - 4 \chi_{-1}(d)}{(n/d)^{2}}.
\]
If $n$ is even, then $r_{6}(n)/n^{2} \leq 8 \zeta(2) \leq
13.2$. On the other hand if $n$ is odd, then
$16 \chi_{-1}(n/d) - 4 \chi_{-1}(d)$ is negative if $n/d \equiv 3 \pmod{4}$
and $16 \chi_{-1}(n/d) - 4 \chi_{-1}(d) \leq 20$ if $n/d \equiv 1 \pmod{4}$.
Thus,
\[
  \frac{r_{6}(n)}{n^{2}} \leq 20 \sum_{\substack{d | n \\ d \equiv 1 \pmod{4}}}
  \frac{1}{d^{2}} \leq 20 \sum_{n=0}^{\infty} \frac{1}{(4n+1)^{2}}.
\]
One can show that the right hand side above is about $21.4966613
\leq \frac{6449}{300}$. We denote by $C_{s}$ a constant so that
$r_{s}(n) \leq C_{s} (n+s)^{\frac{s}{2} - 1}$, and we take
$C_{6} = \frac{6449}{300}$. This proves the base case.

Assume now that $s \geq 6$. We have
\begin{align*}
  r_{s+1}(n) &= r_{s}(n) + 2 \sum_{m=1}^{\lfloor \sqrt{n} \rfloor}
  r_{s}(n-m^{2})\\
  &\leq C_{s} (n+s)^{\frac{s}{2} - 1}
  + 2 C_{s} \sum_{m=1}^{\lfloor \sqrt{n} \rfloor}
  (n+s - m^{2})^{\frac{s}{2} - 1}\\
  &\leq C_{s} (n+s)^{\frac{s}{2} - 1}
  + 2 C_{s} \int_{0}^{\sqrt{n+s+1}} (n+s+1 - x^{2})^{\frac{s}{2} - 1} \, dx\\
  &\leq C_{s} (n+s)^{\frac{s}{2} - 1}
  + 2 C_{s} (n+s+1)^{\frac{s+1}{2} - 1} \int_{0}^{1} (1-u^{2})^{\frac{s}{2} - 1}  \, du.
\end{align*}
We have
\[
  (1-u^{2})^{\frac{s}{2} - 1}
  = e^{\left(\frac{s}{2} - 1\right) \log(1-u^{2})}
  \leq e^{-(s/2 - 1) u^{2}}.
\]
Thus
\[
  2 \int_{0}^{1} (1-u^{2})^{\frac{s}{2} - 1} \, du
  \leq 2 \int_{0}^{\infty} e^{-(s/2 - 1) u^{2}} \, du
  = \sqrt{\frac{\pi}{\frac{s}{2} - 1}},
\]
and
\begin{align*}
  r_{s+1}(n) &\leq C_{s}(n+s)^{\frac{s}{2} - 1}
  + C_{s} (n+s+1)^{\frac{s+1}{2} - 1}
  \left[ \sqrt{\frac{2 \pi}{s - 2}} \right]\\
  &\leq C_{s} (n+s+1)^{\frac{s+1}{2} - 1}
  \left[ \sqrt{\frac{2 \pi}{s-2}} + \frac{(n+s)^{(s/2) - 1}}{(n+s+1)^{\frac{s+1}{2} - 1}} \right].
\end{align*}
Note that the second term inside the brackets above is a decreasing function
of $n$ and is relevant only for $n \geq 1$. It follows that
\begin{align*}
  r_{s+1}(n) &\leq C_{s} (n+s+1)^{\frac{s+1}{2} - 1}
 \cdot \frac{1}{\sqrt{s+1}} \left[ \sqrt{2 \pi} \sqrt{\frac{s+1}{s-2}}
  + \left(\frac{s+1}{s+2}\right)^{s/2 - 1} \right]\\
  &\leq \frac{C_{s} \cdot 4.11}{\sqrt{s+1}} (n+s+1)^{\frac{s+1}{2} - 1}.
\end{align*}
Hence, we may take $C_{s+1} = \frac{4.11}{\sqrt{s+1}} C_{s}$
and so
\[
  C_{s} = \frac{6449}{300} \cdot \frac{4.11^{s-6}}{\sqrt{s!/6!}}
  \leq \frac{3 (4.11)^{s}}{25 \sqrt{s!}}.
\]
\end{proof}

Next, we use Deligne's bound on the Fourier coefficients of a newform
to bound its value.

\begin{lem}
\label{newformbound}
Suppose that $k \geq 7$, $y \geq \frac{1}{2 \pi}$,
and $g(z) = \sum_{n=1}^{\infty} a(n) q^{n}$ with 
$|a(n)| \leq d(n) n^{\frac{k-1}{2}}$.
Then
\[
  |g(x+iy)| \leq \frac{1}{(2 \pi y)^{\frac{k+1}{2}}}
  \Gamma\left(\frac{k+1}{2}\right) \left[ \log\left(\frac{k+1}{2}\right)
  + \gamma + 1 \right],
\]
where $\gamma$ is Euler's constant.
\end{lem}
\begin{proof}
Since the $n$th Fourier coefficient of $g(z)$ is bounded by
$d(n) n^{\frac{k-1}{2}}$, we have that
\[
  |g(x+iy)| \leq \sum_{n=1}^{\infty} d(n) n^{\frac{k-1}{2}} e^{-2 \pi n y}.
\]
If $D(x) = \sum_{n \leq x} d(n)$, then $D(x) \leq x \log(x) + \gamma x + 1
\leq x \log(x) + (\gamma + 1) x$. By partial summation, we have
\begin{align*}
  \sum_{n=1}^{\infty} d(n) n^{\frac{k-1}{2}} e^{-2 \pi n y}
  &= \int_{1}^{\infty} D(x) \left[ 2 \pi y x^{\frac{k-1}{2}} -
  \left(\frac{k-1}{2}\right) x^{\frac{k-3}{2}} \right] e^{-2 \pi x y} \, dx\\
  &\leq 2 \pi y \int_{\frac{k-1}{4 \pi y}}^{\infty} 
  (\log(x) + (\gamma + 1)) x^{\frac{k+1}{2}} e^{-2 \pi x y} \, dx.
\end{align*}
Now, we set $u = 2 \pi x y$, $du = 2 \pi y \, dx$. We get
\begin{align*}
  & 2 \pi y \int_{\frac{k-1}{2}}^{\infty}
  \left(\log\left(\frac{u}{2 \pi y}\right) + (\gamma + 1)\right)
  \left(\frac{u}{2 \pi y}\right)^{\frac{k+1}{2}} e^{-u}
  \, \frac{du}{2 \pi y}\\
  &= \frac{1}{(2 \pi y)^{\frac{k+1}{2}}}
  \int_{\frac{k-1}{2}}^{\infty} \left(\log(u) - \log(2 \pi y) + \gamma
  + 1\right) u^{\frac{k+1}{2}} e^{-u} \, du.
\end{align*}
Since $y \geq \frac{1}{2 \pi}$, $\log(2 \pi y) > 0$ and so we
neglect the term involving it. We get
\[
  \frac{1}{(2 \pi y)^{\frac{k+1}{2}}} 
  \left[ \int_{\frac{k-1}{2}}^{\infty} \log(u) u^{\frac{k+1}{2}} e^{-u} \, du
  + (\gamma + 1) \int_{\frac{k-1}{2}}^{\infty} u^{\frac{k+1}{2}} e^{-u} \, du\right].
\]
If we extend the integrals down to zero, then the negative contribution of\\
$\int_{0}^{1} \log(u) u^{\frac{k+1}{2}} e^{-u} \, du$ is cancelled by that
of $[1.5,2]$ for $k \geq 7$. Thus, we get the bound
\begin{align*}
  |g(z)| &\leq \frac{1}{(2 \pi y)^{\frac{k+1}{2}}}
  \Gamma\left(\frac{k+1}{2}\right) \left[ \psi\left(\frac{k+1}{2}\right)
  + \gamma + 1\right],
\end{align*}
where $\Gamma'(z) = \int_{0}^{\infty} \log(u) u^{z-1} e^{-u} \, du$
and $\psi(z) = \psi(z) = \frac{\Gamma'(z)}{\Gamma(z)}$. The formula
(see equation 6.3.21 on page 258 of \cite{AS})
\[
  \psi(z) = \log(z) - \frac{1}{2z} - \int_{0}^{\infty} \frac{2t \, dt}{(z^{2} + t^{2}) (e^{2 \pi t} - 1)}
\]
shows that $\psi(z) \leq \log(z)$. Thus, we obtain the bound
\[
\frac{1}{(2 \pi y)^{\frac{k+1}{2}}}
  \Gamma\left(\frac{k+1}{2}\right) \left[
  \log\left(\frac{k+1}{2}\right) + \gamma + 1 \right].
\]
\end{proof}

Finally, we will need to understand Petersson inner products of
newforms $f$ with their images under $V(d)$. This is the subject of the
next result.

\begin{lem}
\label{petold}
Suppose that $f(z) = \sum_{n=1}^{\infty} a(n) q^{n} \in S_{k}^{{\rm new}}(\Gamma_{0}(N))$ is a newform. If $p \nmid N$, then
\[
  \langle f, f|V(p) \rangle = \frac{a(p)}{p^{k-1} (p+1)} \langle f, f \rangle.
\]
\end{lem}

Note that the assumption that $f$ has trivial character implies that
the Fourier coefficients of $f$ are real. This fact will be used frequently
in what follows.

\begin{proof}
Rankin proved in \cite{Rankin3} that if $f = \sum a(n) q^{n}$
and $g = \sum b(n) q^{n}$ are cusp forms of weight $k$, then
\[
  \sum_{n \leq x} \frac{a(n) \overline{b(n)}}{n^{k-1}}
  = \frac{(4 \pi)^{k}}{(k-1)!}
  \langle f, g \rangle x + O(x^{3/5}).
\]
We will use this formula to prove the results above. We start by
letting $c = \frac{(4 \pi)^{k}}{(k-1)!}$,
and $p$ be a prime number with $p \nmid N$. Then,
\begin{align*}
  \langle f, f|V(p) \rangle
  &= \lim_{x \to \infty} \frac{1}{c} \cdot \frac{1}{x}
  \sum_{n \leq x} \frac{a(n) a(n/p)}{n^{k-1}}\\
  &= \lim_{x \to \infty} \frac{1}{c} \cdot \frac{1}{x}
  \sum_{pn \leq x} \frac{a(pn) a(n)}{(pn)^{k-1}}\\
  &= \lim_{x \to \infty} \frac{1}{c} \cdot
  \frac{1}{p^{k}} \cdot \frac{1}{\frac{x}{p}}
  \sum_{n \leq \frac{x}{p}} \frac{a(pn) a(n)}{n^{k-1}}\\
  &= \frac{1}{p^{k}} \langle f, f|U(p) \rangle.
\end{align*}
Now, $a(p) f = f|T(p) = f|U(p) +
p^{k-1} f|V(p)$.  It follows that
\begin{align*}
  a(p) \langle f, f \rangle
  &= \langle f, f|T(p) \rangle
  = \langle f, f|U(p) \rangle + p^{k-1} \langle f, f|V(p) \rangle\\
  &= p^{k} \langle f, f|V(p) \rangle + p^{k-1} \langle f, f|V(p) \rangle\\
  &= p^{k-1} (p+1) \langle f, f|V(p) \rangle.
\end{align*}
Thus,
\[
  \langle f, f|V(p) \rangle = \frac{a(p)}{p^{k-1} (p+1)} \langle f, f \rangle.
\]
\end{proof}

\section{Proof of Theorem~\ref{main} and Theorem~\ref{positivity}}
\label{mainsec}

In this section, we will prove the main results. We will first prove
Theorem~\ref{positivity} and then deduce Theorem~\ref{main} from it.

\begin{proof}[Proof of Theorem~\ref{positivity}]
First, for each newform $g$ of level $1$, $2$, or $4$,
we will find a form $\tilde{g}$ with the property that
the coefficient of $g$ in the representation of $\theta^{2k}$
is positive if and only if $\langle \theta^{2k}, \tilde{g} \rangle > 0$.
Each $\tilde{g}$ will be an eigenform of $T_{p}$ for all odd
primes $p$, and will also be an eigenform of $W_{4}$ with the
same eigenvalue as that of $\theta^{2k}$.

Recall the decomposition
\[
  \theta^{2k}(z) = a_{1} E_{k}(z) + a_{2} E_{k}(2z) + a_{3} E_{k}(4z)
  + \sum_{i} c_{i} g_{i}(z) + \sum_{i} d_{i} g_{i}(2z) + \sum_{i} e_{i}
  g_{i}(4z),
\]
where the $g_{i}$ are newforms of level $1$, $2$, or $4$, and 
the $c_{i}, d_{i}, e_{i} \in \R$. If $V$ is an eigenspace for all $T_{n}$ 
(with $n$ odd), then $V$ is also stable under $W_{4}$. Since 
$\theta^{2k} | W_{4} = (-1)^{\frac{k}{2}} \theta^{2k}$, it follows that the 
projection of $\theta^{2k}$ onto $V$ must also have eigenvalue 
$(-1)^{\frac{k}{2}}$ under $W_{4}$.

If $V$ is an eigenspace coming from a newform $g_{i}$ of level $4$, then
$\dim V = 1$. If $c_{i} \ne 0$, then $g_{i} | W_{4} = (-1)^{\frac{k}{2}}$.
In this case, we have
$\langle \theta^{2k}, g \rangle = \langle c_{i} g_{i}, g_{i} \rangle =
c_{i} \langle g_{i}, g_{i} \rangle$ and thus $c_{i} > 0$ if and only if
$\langle \theta^{2k}, g_{i} \rangle > 0$, and so
we set $\tilde{g}_{i} = g_{i}$. Part (i) of Theorem 7 of \cite{AL}
shows that for any newform of level $4$, $g_{i} | W_{4} = -1$,
and hence $c_{i} = 0$ if $k \equiv 0 \pmod{4}$.

If $V$ is an eigenspace coming from a newform $g_{i}$ of level $2$, then
$\dim V = 2$. This vector space decomposes into one-dimensional plus
and minus eigenspaces under the action of $W_{4}$. It follows that
the projection of $\theta^{2k}$ onto $V$ is 
$c_{i} (g_{i} + (-2)^{\frac{k}{2}} \lambda g_{i} | V(2))$,
where $\lambda$ is the Atkin-Lehner eigenvalue of $g_{i}$. Thus,
we set $\tilde{g}_{i} = g_{i} + (-2)^{\frac{k}{2}} \lambda g_{i} | V(2)$.
This form will be orthogonal to any element in the opposite $W_{4}$
eigenspace, since $W_{4}$ is an isometry with respect to the Petersson
inner product. It follows that
$c_{i} > 0$ if and only if $\langle \theta^{2k}, \tilde{g}_{i} \rangle > 0$.

If $V$ is an eigenspace coming from a newform $g_{i}$ of level $1$, then
$\dim V = 3$ and
\begin{align*}
  g_{i} | W_{4} &= 2^{k} g_{i} | V(4)\\
  g_{i} | V(2) | W_{4} &= g_{i} | V(2)\\
  g_{i} | V(4) | W_{4} &= 2^{-k} g_{i}.
\end{align*}
We have that $V = V^{+} \oplus V^{-}$, where $V^{+}$ and $V^{-}$ are the
plus and minus eigenspaces for $W_{4}$. Then $\dim V^{+} = 2$ and
it is spanned by $g_{i} + 2^{k} g_{i} | V(4)$ and $g_{i} | V(2)$.
Also $\dim V^{-} = 1$ and it is spanned by $g_{i} - 2^{k} g_{i} | V(4)$. 
If $k \equiv 0 \pmod{4}$, then the Atkin-Lehner sign is $+1$.
If $k \equiv 2 \pmod{4}$, the Atkin-Lehner sign is $-1$. 

When $k \equiv 2 \pmod{4}$, we set
$\tilde{g}_{i} = g_{i} - 2^{k} g_{i} | V(4)$. This form satisfies
$\tilde{g}_{i} | W_{4} = -\tilde{g}_{i}$, and is again orthogonal to
the form spanning the plus eigenspace for $W_{4}$.

When $k \equiv 0 \pmod{4}$, we set
$\tilde{g}_{i} = g_{i} - \frac{4}{3} a(2) g_{i} | V(2) +
2^{k} g_{i} | V(4)$. This form satisfies $\tilde{g}_{i} | W_{4} = \tilde{g}_{i}$,
and is hence orthogonal to $g_{i} - 2^{k} g_{i} | V(4)$. By
Lemma~\ref{petold} it is orthogonal to $g_{i} | V(2)$.

We have
\begin{align*}
  \langle \theta^{2k}, \tilde{g}_{i} \rangle &=
  \frac{1}{2 \pi} \iint_{\mathbb{H}/\Gamma_{0}(4)} \theta^{2k}(z)
  \overline{\tilde{g}_{i}(z)} y^{k} \, \frac{dx \, dy}{y^{2}}\\
  &= \frac{1}{2 \pi} \sum_{j=1}^{6}
  \int_{-1/2}^{1/2} \int_{\sqrt{1-x^{2}}}^{\infty}
  \left(\theta^{2k} |_{k} \gamma_{j}\right)(x+iy)
  \overline{\tilde{g}_{i} |_{k} \gamma_{j} (x+iy)} y^{k-2} \, dy \, dx.
\end{align*}
Here, the matrices
\[
  \gamma_{1} = \left[ \begin{matrix} 1 & 0\\ 0 & 1 \end{matrix} \right],
  \gamma_{2} = \left[ \begin{matrix} 0 & -1\\ 1 & 0 \end{matrix} \right],
  \gamma_{3} = \left[ \begin{matrix} 0 & -1\\ 1 & 1 \end{matrix} \right],
  \gamma_{4} = \left[ \begin{matrix} 0 & -1\\ 1 & 2 \end{matrix} \right],
  \gamma_{5} = \left[ \begin{matrix} 0 & -1\\ 1 & 3 \end{matrix} \right],
  \gamma_{6} = \left[ \begin{matrix} 1 & 0 \\ 2 & 1 \end{matrix} \right]
\]
are a set of representatives for the right cosets of $\Gamma_{0}(4)$ in
$\SL_{2}(\Z)$. 

{\bf Term 1}: This is the contribution from the cusp at infinity. In
particular, it is the $j = 1$ term in the above sum. We split
this term into two parts: $\{ x+iy : -1/2 \leq x \leq 1/2, y \geq 1 \}$,
and $\{ x+iy : -1/2 \leq x \leq 1/2, \sqrt{1-x^{2}} \leq y \leq 1 \}$.

Write
\[
  \tilde{g}_{i}(z) = \sum_{n=1}^{\infty} a(n) q^{n}.
\]
Applying the Deligne bound to each of the various possible forms
of $\tilde{g_{i}}$, we see that in all cases $|a(n)| \leq \frac{17}{3}
d(n) n^{\frac{k-1}{2}}$.

The first part is
\[
  \frac{1}{2 \pi} \int_{1}^{\infty} \int_{-1/2}^{1/2} \left(\sum_{m=0}^{\infty}
  r_{2k}(m) e^{-2 \pi m y} e^{2 \pi i m x}\right)
  \left(\sum_{n=1}^{\infty} \overline{a(n)} e^{-2 \pi n y} e^{-2 \pi i n x}
  \right) y^{k-2} \, dx \, dy.
\]
Since the Fourier series representations converge uniformly on 
compact subsets of these regions, we can invert the summations and the 
integrals and obtain
\[
  \frac{1}{2 \pi} \sum_{m=0}^{\infty} \sum_{n=1}^{\infty}
  r_{2k}(m) \overline{a(n)} \int_{-1/2}^{1/2} 
  \int_{1}^{\infty} y^{k-2} e^{-2 \pi (m+n)y} e^{2 \pi i (m-n) x} \, dy \, dx.
\]
The integral over $-1/2 \leq x \leq 1/2$ is zero unless $m = n$, in which
case it is $1$. We set $u = 4 \pi n y$, $du = 4 \pi n \, dy$ and this
gives
\[
  \frac{2}{(4 \pi)^{k}} \sum_{n=1}^{\infty}
  \frac{r_{2k}(n) \overline{a(n)}}{n^{k-1}}
  \int_{4 \pi n}^{\infty} u^{k-2} e^{-u} \, du.
\]
  
We now split this sum into several ranges. The main contribution comes
from $n = 1$. We have $a(1) = 1$ and $r_{2k}(1) = 4k$. This term is
\begin{align*}
  \frac{8k}{(4 \pi)^{k}} \int_{4 \pi}^{\infty} u^{k-2} e^{-u} \, du
  &= \frac{8k}{(4 \pi)^{k}} \left[
  \int_{0}^{\infty} u^{k-2} e^{-u} \, du - \int_{0}^{4 \pi} u^{k-2} e^{-u} \, du  \right]\\
  & \geq \frac{8k}{(4 \pi)^{k}} \left[ (k-2)! - (4 \pi)^{k-1} e^{-4 \pi} \right],
\end{align*}
for $k \geq 15$, since if $k > 4 \pi + 2$, $u^{k-2} e^{-u}$ is increasing on 
$[0,4 \pi]$.

The second range is $2 \leq n \leq 2500$. Here we explicitly compute
the polynomials $r_{2k}(n)$ (using the algorithm in the proof of part
1 of Lemma~\ref{simplebound}). Part 2 of Lemma~\ref{simplebound} shows
that $\frac{r_{2k}(n)}{n^{\frac{k-1}{2}}}$ is a decreasing function of
$k$, provided $k \geq 1456$.

The third range is $2500 \leq n \leq \frac{k}{2 \pi} \log(2k)$. In this range,
we use the bound from part 3 of Lemma~\ref{simplebound}, the Deligne bound, 
$d(n) \leq 2 \sqrt{n}$, and we obtain that
\begin{align*}
  \left| \frac{r_{2k}(n) \overline{a(n)}}{n^{k-1}} \right|
  &\leq \frac{34 (4.11)^{2k}}{25 \sqrt{(2k)!}} \cdot \sqrt{n}
  \cdot \frac{(n+2k)^{k-1}}{n^{\frac{k-1}{2}}}\\
  &\leq \frac{34}{25} \frac{(4.11)^{2k}}{\sqrt{(2k)!}} \cdot
  \sqrt{\frac{k}{2 \pi} \log(2k)} \cdot \left(\sqrt{n} + \frac{2k}{\sqrt{n}}\right)^{k-1}.
\end{align*}
The function $f(x) = \left(x + \frac{2k}{x}\right)^{k-1}$ is
decreasing for $x < \sqrt{2k}$ and increasing after that. We have that
$f(50) = f(\frac{2k}{50})$ and $\frac{2k}{50} \geq \sqrt{\frac{k}{2 \pi}
 \log(2k)}$ if $k \geq 724$. Thus, we have the bound
\[
  \frac{68}{25} \frac{(4.11)^{2k}}{(4 \pi)^{k} \sqrt{(2k)!}}
  \cdot \frac{k^{3/2}}{(2 \pi)^{3/2}} \log^{3/2}(2k) \cdot \left(
  50 + \frac{2k}{50}\right)^{k-1} \cdot (k-2)!,
\]
valid provided $k \geq 724$. For $k \leq 724$, we use the larger of
the values of $f$ at $x = 50$ and $x = \sqrt{\frac{k}{2 \pi} \log(2k)}$.

The fourth and final range is $n \geq \frac{k}{2 \pi} \log(2k)$. In this
range we use the decay of the integral $\int_{4 \pi n}^{\infty} u^{k-2}
e^{-u} \, du$. We have that $u \geq 2k \log(2k)$ and so
$u^{k-2} e^{-u} \leq e^{-u/2}$ and so the integral is bounded by 
$2 e^{-2 \pi n}$. Bounding $\overline{a(n)}$ and $r_{2k}(n)$ as before, we
have that the contribution from this range is at most
\begin{align*}
  & \frac{34}{3 (4 \pi)^{k}} \sum_{n=\frac{k}{2 \pi} \log(2k)}^{\infty}
  \frac{2 n^{\frac{k}{2}}}{n^{k-1}} \cdot \left(\frac{3}{25}
  \cdot \frac{(4.11)^{2k}}{\sqrt{(2k)!}}\right) (n+2k)^{k-1} \cdot
  2 e^{-2 \pi n}.
\end{align*}
We write $\frac{(n+2k)^{k-1}}{n^{k-1}}$ as $\left(1 + \frac{2k}{n}\right)^{k-1}$.
If $k \geq 40$, $1 + \frac{2k}{n} \leq 3.87$ and we get
\[
  \frac{136}{25 (4 \pi)^{k}} \cdot \frac{(4.11)^{2k} (3.87)^{k-1}}{\sqrt{(2k)!}}
  \sum_{n=\frac{k}{2 \pi} \log(2k)}^{\infty} n^{\frac{k}{2}} e^{-2 \pi n}.
\]
If $a_{n} = n^{\frac{k}{2}} e^{-2 \pi n}$, then we have
\[
  \frac{a_{n+1}}{a_{n}} \leq \left(1 + \frac{1}{n}\right)^{\frac{k}{2}}
  e^{-2 \pi} \leq e^{\frac{k}{2n} - 2 \pi} \leq e^{-2 \pi + \frac{\pi}{\log(2k)}}\\
  \leq e^{-5.6}.
\]  
Thus, we get the bound
\[
    \frac{136}{25 (4 \pi)^{k}} \cdot \frac{(4.11)^{2k} (3.87)^{k-1}}{\sqrt{(2k)!}} \cdot \left(\frac{k}{2 \pi} \log(2k)\right)^{\frac{k}{2}}
  (2k)^{-k} \cdot \frac{1}{1 - e^{-5.6}},
\]
valid if $k \geq 40$. 

The second part of the contribution from the cusp at infinity is
\[
  \frac{1}{2 \pi} \int_{-1/2}^{1/2} \int_{\sqrt{1-x^{2}}}^{1}
  \theta^{2k}(x+iy) \overline{g(x+iy)} y^{k-2} \, dy \, dx.
\]
In this region we use Lemma~\ref{newformbound} to bound $g(x+iy)$,
and we use that
\[
  |\theta(z)| \leq 1 + 2 \sum_{n=1}^{\infty} e^{-2 \pi n^{2} y}
  \leq 1.008667
\]
for $y \geq \sqrt{3}/2$. This gives the bound
\[
  \frac{\Gamma\left(\frac{k+1}{2}\right) \left[\log\left(\frac{k+1}{2}\right) + \gamma + 1 \right] (1.008667)^{2k}}{(2 \pi)^{\frac{k+3}{2}}}
  \int_{-1/2}^{1/2} \int_{\sqrt{1-x^{2}}}^{1} y^{\frac{k-5}{2}} \, dy \, dx.
\]
The double integral above is less than or equal to
$\int_{-1/2}^{1/2} \int_{0}^{1} y^{\frac{k-5}{2}} \, dy \, dx
= \frac{2}{k-3}$. Hence, we get the bound
\[
  \frac{34 \Gamma\left(\frac{k+1}{2}\right) \left[\log\left(\frac{k+1}{2}\right) + \gamma + 1 \right] (1.008667)^{2k}}{3 (k-3) (2 \pi)^{\frac{k+3}{2}}},
\]
valid for $k \geq 7$.
  
{\bf Term 2}: This is the contribution of the cusp at zero,
and in particular the contributions from the terms involving 
$\gamma_{2}$, $\gamma_{3}$, $\gamma_{4}$, and $\gamma_{5}$.
We have
\[
  \theta^{2k} | W_{4} = (-1)^{\frac{k}{2}} \theta^{2k} \text{ and }
  \tilde{g}_{i} | W_{4} = (-1)^{\frac{k}{2}} \tilde{g}_{i}.
\]
Translating this into Fourier expansions gives
\[
  \theta^{2k} | \left[ \begin{matrix} 0 & -1 \\ 1 & 0 \end{matrix} \right]
  = \frac{(-1)^{\frac{k}{2}}}{2^{k}} \theta^{2k}\left(\frac{z}{4}\right),
  \qquad
  \tilde{g}_{i} | \left[ \begin{matrix} 0 & -1 \\ 1 & 0 \end{matrix} \right]
  = \frac{(-1)^{\frac{k}{2}}}{2^{k}} \tilde{g_{i}}\left(\frac{z}{4}\right).
\]
Thus, the contribution from these four terms is
\[
  \frac{1}{(2 \pi) \cdot 4^{k}} \sum_{j=0}^{3} \int_{-1/2}^{1/2}
  \int_{\sqrt{1-x^{2}}}^{\infty} \theta^{2k}\left(\frac{x + j + iy}{4}\right)
  \overline{\tilde{g}_{i}\left(\frac{x + j + iy}{4}\right)} \, y^{k-2}
  \, dy \, dx.
\]
We set $u = x/4$ and $v = y/4$ in the integrand and obtain
\[
  \frac{1}{2 \pi} \sum_{j=0}^{3} \int_{-1/8}^{1/8} \int_{\frac{\sqrt{1-16u^{2}}}{4}}^{\infty} \theta^{2k}\left(u + iv + \frac{j}{4}\right)
\overline{\tilde{g}_{i}\left(u + iv + \frac{j}{4}\right)} v^{k-2} \, dv \, du.
\]
We break this into two terms. The first term consists of those pieces with
$v \leq 1$. The smallest value $v$ takes on this piece is $\sqrt{3}/8$ and
since $\sqrt{3}/8 > \frac{1}{2 \pi}$, we may use Lemma~\ref{newformbound} to
bound the contribution. This yields
\[
  |\tilde{g}_{i}(u+iv)| \leq \frac{17}{3 \cdot (2 \pi v)^{\frac{k+1}{2}}}
  \Gamma\left(\frac{k+1}{2}\right) \left[ \log\left(\frac{k+1}{2}\right)
  + \gamma + 1 \right].
\]
We also have
\[
  |\theta(u+iv)| \leq 1 + 2 \sum_{n=1}^{\infty} e^{-2 \pi n^{2} v} \leq
  1.52182
\]
for $v \geq \sqrt{3}/8$. The contribution of these terms is therefore
bounded by  
\[
  \frac{17 \cdot (1.52182)^{2k}}{3 \cdot (2 \pi)^{\frac{k+3}{2}}}
  \Gamma\left(\frac{k+1}{2}\right) \left[\log\left(\frac{k+1}{2}\right)
  + \gamma + 1 \right] \sum_{j=0}^{3} \int_{-1/8}^{1/8} 
\int_{\frac{\sqrt{1 - 16u^{2}}}{4}}^{1} v^{\frac{k-5}{2}} \, dv
  \, du
\]
The sum of double integrals is bounded by $\int_{-1/2}^{1/2} \int_{0}^{1}
v^{\frac{k-5}{2}} \, dv = \frac{2}{k-3}$ and we get the bound
\[  
  \frac{34 \cdot (1.52182)^{2k}}{3 \cdot (2 \pi)^{\frac{k+3}{2}} \cdot (k-3)}
  \Gamma\left(\frac{k+1}{2}\right) \left[\log\left(\frac{k+1}{2}\right)
  + \gamma + 1 \right],
\]
on the part where $v \leq 1$, valid for $k \geq 7$.

The second term consists of those pieces with $v \geq 1$. This gives
\[
  \frac{1}{2 \pi} \int_{-1/2}^{1/2} \int_{1}^{\infty}
  \theta^{2k}(u + iv) \overline{\tilde{g}_{i}(u + iv)} v^{k-2} \, dv \, du.
\]
This is exactly the same as the contribution of the first
part of the cusp at infinity! 

{\bf Term 3}: This is the contribution of the cusp at $1/2$ corresponding
to the matrix $\gamma_{6}$. We must understand the Fourier expansion
of $\tilde{g}_{i} | \gamma_{6}$. Since $\gamma_{6} \in \Gamma_{0}(2)$,
terms of level 1 or level 2 are not affected.

If $g$ is a newform of level $4$, then since $\gamma_{6}$ is not in 
$\Gamma_{0}(4)$, we have that $g \mapsto g + g | \gamma_{6}$ is the trace map
from $S_{k}(\Gamma_{0}(4))$ to $S_{k}(\Gamma_{0}(2))$. Since newforms are in the
kernel of the trace map (by Theorem 4 of \cite{Li}), it follows that $g + g | \gamma_{6} = 0$
and so $g | \gamma_{6} = -g$.

If $g$ is a newform of level $2$, we have
\[
  g | V(2) | \gamma_{6} = 2^{-k/2} g | \left[ \begin{matrix}
  2 & 0 \\ 0 & 1 \end{matrix} \right] \left[ \begin{matrix}
  1 & 0 \\ 2 & 1 \end{matrix} \right]\\
  = 2^{-k/2} g | \left[ \begin{matrix}
  0 & -1 \\ 2 & 0 \end{matrix} \right] \left[ \begin{matrix}
  1 & 0 \\ -2 & 1 \end{matrix} \right] \left[ \begin{matrix}
  1 & 1/2 \\ 0 & 1 \end{matrix} \right].
\]
The first matrix is the Atkin-Lehner involution of level 2, of which $g$ is
an eigenform. The second matrix is in $\Gamma_{0}(2)$ and the third matrix
does not affect the size of the Fourier coefficients at infinity.
It follows that the $n$th Fourier coefficient of
$g | V(2) | \gamma_{6}$ is bounded by $2^{-k/2} d(n) n^{\frac{k-1}{2}}$.

If $g$ is a newform of level $1$, we have
\[
  g | V(4) | \gamma_{6} = 2^{-k} g \left[ \begin{matrix}
  4 & 0 \\ 0 & 1 \end{matrix} \right] \left[ \begin{matrix}
  1 & 0 \\ 2 & 1 \end{matrix} \right]
  = 2^{-k} g | \left[ \begin{matrix} 2 & 1 \\ 1 & 1 \end{matrix} \right]
  \left[ \begin{matrix} 2 & -1 \\ 0 & 2 \end{matrix} \right]
  = 2^{-k} g(z - 1/2).
\]
Thus, the $n$th Fourier coefficient of $g | V(4) | \gamma_{6}$ is bounded
by $2^{-k} d(n) n^{\frac{k-1}{2}}$. It follows that for any $\tilde{g}_{i}$,
the $n$th coefficient of $\tilde{g}_{i} | \gamma_{6}$ is bounded by
$\frac{14}{3} d(n) n^{\frac{k-1}{2}}$.

Now, $\theta^{2k} | \gamma_{6} = 2^{2k}
\frac{\eta(4z)^{4k}}{\eta(2z)^{4k}}$.  The form $F(z) =
\frac{\eta(4z)^{8}}{\eta(2z)^{4}} \in M_{2}(\Gamma_{0}(4))$ and
satisfies
\[
  F(z) = \sum_{n \text{ odd }} \sigma(n) q^{n}.
\]
Thus, for $y \geq \sqrt{3}/2$, $|F(z)| \leq e^{-2 \pi y}
\left(\sum_{n \text{ odd }} \sigma(n) e^{-2 \pi (n-1) y}\right)
\leq 1.0001 e^{-2 \pi y}$ and so
\[
  |\theta^{2k} | \gamma_{6}| \leq 2^{2k} (1.0001)^{k/2} e^{-k \pi y}
\]
for $y \geq \sqrt{3}/2$. The contribution of the cusp at $1/2$ is therefore
\[
  \frac{1}{2 \pi} \int_{-1/2}^{1/2} \int_{\sqrt{1-x^{2}}}^{\infty}
  \theta^{2k} | \gamma_{6}(x+iy) \overline{\tilde{g}_{i} | \gamma_{6}(x+iy)}
  y^{k-2} \, dy \, dx.
\]
By Lemma~\ref{newformbound}, we have 
\[
  |\tilde{g}_{i} | \gamma_{6}(x+iy)|
\leq \frac{14}{3} \frac{1}{(2 \pi)^{\frac{k+1}{2}}}
\Gamma\left(\frac{k+1}{2}\right) \left[ \log\left(\frac{k+1}{2}\right)
+ \gamma + 1 \right] \cdot \frac{1}{y^{\frac{k+1}{2}}}. 
\]
This gives the bound
\[
  \frac{14 \cdot 2^{2k} \cdot (1.0001)^{k/2}}{3 (2 \pi)^{\frac{k+3}{2}}} 
  \Gamma\left(\frac{k+1}{2}\right) \left[ \log\left(\frac{k+1}{2}\right)
  + \gamma + 1 \right] \int_{0}^{\infty}
  y^{\frac{k-5}{2}} e^{-k \pi y} \, dy.
\]
The integral above is $\frac{1}{(k \pi)^{\frac{k-3}{2}}} \Gamma\left(\frac{k-3}{2}\right)$, and so the bound on this term is
\[
  \frac{14 \cdot 2^{2k} \cdot (1.0001)^{k/2}}{3 (2 \pi)^{\frac{k+3}{2}}
  (k \pi)^{\frac{k-3}{2}}}
  \Gamma\left(\frac{k+1}{2}\right) \Gamma\left(\frac{k-3}{2}\right) 
  \left[ \log\left(\frac{k+1}{2}\right)
  + \gamma + 1 \right]
\]
and is valid for $k \geq 7$.

After dividing each term above by $\frac{(k-2)!}{(4 \pi)^{k}}$,
the main term is increasing linearly (it is about $16k$), and
each other term decreases exponentially. The most troublesome term
is the term from the third range of values of $n$ from the cusp at infinity,
and (after dividing by $\frac{(k-2)!}{(4 \pi)^{k}}$) is asymptotic to
$c_{1} c_{2}^{k} k^{1/4} \ln(k)^{3/2}$, where $c_{2} \approx 0.918$, but 
$c_{1} \approx 1.69 \cdot 10^{543}$. This term is larger than the main term only
when $k \geq 14000$.

For this reason, we must explicitly calculate our bounds for $k <
14000$. In this range, we refine our estimate of the troublesome term
by using the exact values of the incomplete $\Gamma$-function $\int_{4
  \pi n}^{\infty} u^{k-2} e^{-u} \, du$. Also, for $k \leq 2550$, we
compute the values of $r_{2k}(n)$ explicitly for $2 \leq n \leq 2500$
and use these in our bounds. For $k \geq 2552$, we use part 2 of
Lemma~\ref{simplebound}.

Finally, for $k \leq 194$, our numerical bounds are not sufficient
and we use Magma to explicitly compute the decomposition
of $\theta^{2k}$ as in equation \eqref{decomp} and find that the constants
$c_{i}$ are non-negative. 
\end{proof}

\begin{proof}[Proof of Theorem~\ref{main}]
First, assume that $n$ is odd. Considering the coefficient of $q$
on both sides of \eqref{decomp}, we obtain
\[
  r_{2k}(1) = 4k = \frac{2k (-1)^{k/2}}{(2^{k} - 1) B_{k}} + \sum_{i} c_{i}.
\]
By Theorem~\ref{positivity}, we have
\[
  \sum_{i} |c_{i}| = \sum_{i} c_{i} = 4k - \frac{2k (-1)^{k/2}}{(2^{k} - 1) B_{k}}.
\]
Deligne's bound on the $n$th coefficient of $g_{i}(z)$ is bounded by 
$d(n) n^{\frac{k-1}{2}}$. Plugging this bound into the decomposition and
using the fact that the coefficients of $q^{n}$ in $g_{i}(2z)$ and
$g_{i}(4z)$ are zero if $n$ is odd gives the desired bound on
the cusp form contribution to $\theta^{2k}(z)$.

Now, suppose that $k/2$ is odd and $n$ is even. Then $k \equiv 2 \pmod{4}$.
We represent the decomposition of the cusp form part of $\theta^{2k}(z)$ as 
\[
  C(z) = \sum_{i} r_{i} \left(f_{i}(z) - 2^{k} f_{i}(z) | V(4)\right)
  + \sum_{i} s_{i} \left(g_{i}(z) - 2^{\frac{k}{2}} \lambda_{i} g_{i}(z)
  | V(2)\right)
  + \sum_{i} t_{i} h_{i}(z).
\]
Here, the $f_{i}(z)$, $g_{i}(z)$ and $h_{i}(z)$ are the newforms of levels
$1$, $2$, and $4$ respectively, and $\lambda_{i}$ is the Atkin-Lehner
eigenvalue of $g_{i}(z)$. One can see that the $n$th
coefficients of $f_{i}(z) - 2^{k} f_{i}(z) | V(4)$ and
$g_{i}(z) - 2^{\frac{k}{2}} \lambda_{i} g_{i}(z) | V(2)$ are
bounded by $3 d(n) n^{\frac{k-1}{2}}$. Thus, for even $n$, we obtain the
bound
\[
  \left(\sum_{i} 3 r_{i} + 3 s_{i}\right) d(n) n^{\frac{k-1}{2}}.
\]
We will show that $\sum_{i} 3 r_{i} + 3 s_{i} < 4k - \frac{2k}{(2^{k} - 1) B_{k}}$.

To compute the constant $\sum_{i} 3 r_{i} + 3 s_{i}$, we will compute the trace 
of $C(z)$ to $S_{k}(\Gamma_{0}(2))$, given by ${\rm Tr}(C) := 
C(z) + C(z) | \left[ \begin{matrix} 1 & 0 \\ 2 & 1 \end{matrix} \right]$. 
Straight-forward, but somewhat lengthy computations show that
\begin{align*}
  {\rm Tr}(f_{i}(z) - 2^{k} f_{i}(z) | V(4)) &= 3 f_{i}(z) - 2 a_{i}(2)
  f_{i}(z) | V(2)\\
  {\rm Tr}(g_{i}(z) - 2^{\frac{k}{2}} \lambda_{i} g_{i}(z) | V(2)
  &= 3 g_{i}(z)\\
  {\rm Tr}(h_{i}(z)) &= 0.
\end{align*}
It follows from these formulas that $\sum_{i} 3 r_{i} + 3 s_{i}$ is the
coefficient of $q$ in ${\rm Tr}(C)$. We have that
\begin{align*}
  C &= \theta^{2k} + \frac{1}{2^{k} - 1} E_{k}(z)
  - \frac{2^{k}}{2^{k} - 1} E_{k}(4z)\\
  {\rm Tr}(C) &= {\rm Tr}(\theta^{2k})
  - \frac{(-1)^{k/2}}{2^{k} - 1} {\rm Tr}(E_{k}(z))
  - \frac{2^{k}}{2^{k} - 1} {\rm Tr}(E_{k}(4z))\\
  &= \left(\theta^{2k} + 4^{k} \frac{\eta^{4k}(4z)}{\eta^{2k}(2z)}\right)
  + \frac{2}{2^{k} - 1} E_{k}(z)
  - \frac{2^{k}}{2^{k} - 1} \left((1 + 2^{1-k}) E_{k}(z) | V(2) - 2^{-k} E_{k}(z)\right).
\end{align*}
Taking the coefficient of $q$ on both sides of the preceding equation gives
\[
  \sum_{i} 3 r_{i} + 3 s_{i} = 4k - \frac{6k}{(2^{k} - 1) B_{k}}
  < 4k - \frac{2k}{(2^{k} - 1) B_{k}}
\] 
since $k \equiv 2 \pmod{4}$ and hence $B_{k} > 0$. This proves 
Theorem~\ref{main} in the case that $k \equiv 2 \pmod{4}$ and $n$ is even.
\end{proof}

\section{Final remarks}
\label{finalsec}

It is natural to consider if Theorem~\ref{main} is true in other cases.
When $k \equiv 0 \pmod{4}$ and $n$ is even, the main issue is that
if $g_{i}$ is a level 1 eigenform and
\[
  \tilde{g}_{i} = g_{i} - \frac{4}{3} a(2) g_{i} | V(2) + 2^{k} g_{i} | V(4)
  = \sum_{n=1}^{\infty} c(n) q^{n}
\]
then the best possible bound on the Fourier coefficients of
$\tilde{g}_{i}$ is $|c(n)| \leq 2 d(n) n^{\frac{k-1}{2}}$. In order
for this bound to come close to being achieved, it is necessary for
$|a(2)|$, the absolute value of the second coefficient in $g_{i}$, to
be close to $2 \cdot 2^{\frac{k-1}{2}}$. Serre proved in 1997 (see
\cite{Serre}) that if $p$ is a fixed prime, the $p$th coefficients of
newforms become equidistributed (along any sequence of weights and
levels whose sum tends to infinity, where the levels are not multiples
of $p$). It follows from this that there will be level 1 eigenforms
with $|a(2)|$ arbitrarily close to $2 \cdot 2^{\frac{k-1}{2}}$, but
also that there will be few such forms.  One approach to extending
Theorem~\ref{main} to the case when $k \equiv 0 \pmod{4}$ is to use
the equidistribution of the numbers $|a(2)|$.

It is also natural to consider the problem of deriving a sharp bound
in the case that $k$ is odd. In the case when $k$ is even, the
contribution from the cusp at zero is (up to a fairly small error) the
same as the contribution at the cusp at infinity, since both
$\theta^{2k}$ and the newforms are eigenforms of the Atkin-Lehner
involution $W_{4}$. However, when $k$ is odd, the newforms are not
eigenforms of $W_{4}$ any longer. This means that the contribution of
the cusp at zero is (up to some small error) the contribution of the
cusp at infinity times some complex number $\lambda$ of absolute value
$1$. This complex number is related to the coefficient of $q^{4}$ of
the relevant eigenform $g_{i}$. A similar result could be proven
provided one could rule out the possibility that $\lambda$ is close to
$-1$. In fact, the analogue of Theorem~\ref{positivity} is false for $k = 17$, 
although this seems to be a consequence of the smallness of the weight, rather 
than a value of $\lambda$ too close to $-1$.

For half-integral values of $k$ (corresponding to representations of $n$
as the sum of an odd number of squares), the question is still interesting.
In this case, the coefficients of the cusp forms involve square roots of
central critical $L$-values of quadratic twists of forms of level $1$ and level
$2$. The analogue of Deligne's theorem in this case would be optimal 
subconvexity bounds on these $L$-values, currently attainable only under the
assumption of the generalized Riemann hypothesis.

\bibliographystyle{amsplain}
\bibliography{refs}

\end{document}